\documentclass[reqno]{amsart}
\usepackage{amsfonts}

\newtheorem{theorem}{Theorem}
\theoremstyle{plain}

\newtheorem{corollary}{Corollary}

\newtheorem{definition}{Definition}
\newtheorem{example}{Example}

\newtheorem{lemma}{Lemma}

\newtheorem{proposition}{Proposition}

\numberwithin{equation}{section}
\numberwithin{theorem}{section}
\numberwithin{algorithm}{section}
\numberwithin{axiom}{section}
\numberwithin{case}{section}
\numberwithin{claim}{section}
\numberwithin{conclusion}{section}
\numberwithin{condition}{section}
\numberwithin{conjecture}{section}
\numberwithin{corollary}{section}
\numberwithin{criterion}{section}
\numberwithin{definition}{section}
\numberwithin{example}{section}
\numberwithin{exercise}{section}
\numberwithin{lemma}{section}
\numberwithin{notation}{section}
\numberwithin{problem}{section}
\numberwithin{proposition}{section}
\numberwithin{remark}{section}
\numberwithin{solution}{section}

\input{tcilatex}

\begin{document}
\title{Improved Moser-Trudinger-Onofri inequality under constraints}
\author{Sun-Yung A. Chang}
\address{Department of Mathematics, Princeton University, Fine Hall,
Washington Road, Princeton, NJ 08544}
\email{chang@math.princeton.edu}
\author{Fengbo Hang}
\address{Courant Institute, New York University, 251 Mercer Street, New York
NY 10012}
\email{fengbo@cims.nyu.edu}

\begin{abstract}
A classical result of Aubin states that the constant in
Moser-Trudinger-Onofri inequality on $\mathbb{S}^{2}$ can be imporved for
functions with zero first order moments of the area element. We generalize
it to higher order moments case. These new inequalities bear similarity to a
sequence of Lebedev-Milin type inequalities on $\mathbb{S}^{1}$ coming from
the work of Grenander-Szego on Toeplitz determinants (as pointed out by
Widom). We also discuss the related sharp inequality by a perturbation
method.
\end{abstract}

\maketitle

\section{Introduction\label{sec1}}

Let $\left( M,g\right) $ be a smooth compact Riemann surface without
boundary. For an integrable function $u$ on $M$, we denote%
\begin{equation}
\overline{u}=\frac{1}{\mu \left( M\right) }\int_{M}ud\mu .  \label{eq1.1}
\end{equation}%
Here $\mu $ is the measure associated with the Riemannian metric $g$.

The classical Moser-Trudinger inequality (see \cite{ChY2, F,M}) tells us
that for every $u\in H^{1}\left( M\right) \backslash \left\{ 0\right\} $
with $\overline{u}=0$, we have%
\begin{equation}
\int_{M}e^{4\pi \frac{u^{2}}{\left\Vert \nabla u\right\Vert _{L^{2}\left(
M\right) }^{2}}}d\mu \leq c\left( M,g\right) .  \label{eq1.2}
\end{equation}%
Here $c\left( M,g\right) $ is a positive constant independent of $u$.

A direct consequence of (\ref{eq1.2}) is the following
Moser-Trudinger-Onofri inequality: for every $u\in H^{1}\left( M\right) $
with $\overline{u}=0$, we have%
\begin{equation}
\log \int_{M}e^{2u}d\mu \leq \frac{1}{4\pi }\left\Vert \nabla u\right\Vert
_{L^{2}\left( M\right) }^{2}+c_{1}\left( M,g\right) .  \label{eq1.3}
\end{equation}%
We remark that the inequality (\ref{eq1.3}) has attracted more interest than
the original inequality (\ref{eq1.2}) due to its close relation to Gauss
curvature equation and spectral geometry through the classical Polyakov
formula (see for example \cite{On,OsPS}).

On the standard sphere, it is found in \cite[corollary 2 on p159]{A} that
for $u\in H^{1}\left( \mathbb{S}^{2}\right) $ with $\overline{u}=0$ and $%
\int_{\mathbb{S}^{2}}x_{i}e^{2u\left( x\right) }d\mu \left( x\right) =0$ for 
$i=1,2,3$, the constant $\frac{1}{4\pi }$ in (\ref{eq1.3}) can be lowered
i.e. for any $\varepsilon >0$, we have%
\begin{equation}
\log \left( \frac{1}{4\pi }\int_{\mathbb{S}^{2}}e^{2u}d\mu \right) \leq
\left( \frac{1}{8\pi }+\varepsilon \right) \left\Vert \nabla u\right\Vert
_{L^{2}}^{2}+c_{\varepsilon }.  \label{eq1.4}
\end{equation}%
Here $c_{\varepsilon }$ is a constant depending on $\varepsilon $ only.

A closely related question is to find the best constant in (\ref{eq1.3}) and
(\ref{eq1.4}). In \cite{On}, the best constant $c_{1}\left( M,g\right) $ for
(\ref{eq1.3}) is found on the standard $\mathbb{S}^{2}$. More precisely it
is shown that for $u\in H^{1}\left( \mathbb{S}^{2}\right) $ with $\overline{u%
}=0$, we have%
\begin{equation}
\log \left( \frac{1}{4\pi }\int_{\mathbb{S}^{2}}e^{2u}d\mu \right) \leq 
\frac{1}{4\pi }\left\Vert \nabla u\right\Vert _{L^{2}}^{2}.  \label{eq1.5}
\end{equation}

For (\ref{eq1.4}), it is proved recently in \cite{GuM} that the best
constant $c_{\varepsilon }$ is $0$. In other words, for $u\in H^{1}\left( 
\mathbb{S}^{2}\right) $ with $\overline{u}=0$ and $\int_{\mathbb{S}%
^{2}}x_{i}e^{2u\left( x\right) }d\mu \left( x\right) =0$ for $i=1,2,3$, we
have%
\begin{equation}
\log \left( \frac{1}{4\pi }\int_{\mathbb{S}^{2}}e^{2u}d\mu \right) \leq 
\frac{1}{8\pi }\left\Vert \nabla u\right\Vert _{L^{2}}^{2}.  \label{eq1.6}
\end{equation}%
This confirms a conjecture in \cite{ChY1}.

To motivate our discussion, let us look at some research on $\mathbb{S}^{1}$
which has similar spirit as above. For convenience we let $D$ be the unit
disk in $\mathbb{R}^{2}$. For any $u\in H^{1}\left( D\right) $ with $\int_{%
\mathbb{S}^{1}}ud\theta =0$, the Lebedev-Milin inequality (see \cite[chapter
5]{D}) tells us%
\begin{equation}
\log \left( \frac{1}{2\pi }\int_{\mathbb{S}^{1}}e^{u}d\theta \right) \leq 
\frac{1}{4\pi }\left\Vert \nabla u\right\Vert _{L^{2}\left( D\right) }^{2}.
\label{eq1.7}
\end{equation}%
This should be compared to (\ref{eq1.5}).

On the other hand, as observed in \cite{Wi}, we have a sequence of
Lebedev-Milin type inequalities following from the work of Grenander-Szego 
\cite{GrS} on Toeplitz determinants. More precisely for any integer $m\geq 0$%
, $u\in H^{1}\left( D\right) $ with $\int_{\mathbb{S}^{1}}ud\theta =0$ and $%
\int_{\mathbb{S}^{1}}e^{u}e^{ik\theta }d\theta =0$ for $k=1,\cdots ,m$, we
have%
\begin{equation}
\log \left( \frac{1}{2\pi }\int_{\mathbb{S}^{1}}e^{u}d\theta \right) \leq 
\frac{1}{4\pi \left( m+1\right) }\left\Vert \nabla u\right\Vert
_{L^{2}\left( D\right) }^{2}.  \label{eq1.8}
\end{equation}%
For $m=0$, (\ref{eq1.8}) is just (\ref{eq1.7}). For $m=1$, (\ref{eq1.8}) is
proved in \cite[section 2]{OsPS}. These inequalities should be compared to (%
\ref{eq1.6}). Note that $\cos k\theta $ and $\sin k\theta $ are
eigenfunctions of $-\Delta _{\mathbb{S}^{1}}$ with eigenvalue $k^{2}$. So (%
\ref{eq1.8}) actually tells us we can improve the coefficient of $\left\Vert
\nabla u\right\Vert _{L^{2}\left( D\right) }^{2}$ further if $e^{u}$ is
perpendicular to more eigenfunctions of $-\Delta _{\mathbb{S}^{1}}$. For a
while, people wonder whether we have similar improvements of (\ref{eq1.4})
or (\ref{eq1.6}) on $\mathbb{S}^{2}$. The main aim of this note, as stated
in Theorem \ref{thm1.1} below, is to confirm this guess.

To state the main results, we need some notations. For any nonnegative
integer $k$, we denote%
\begin{eqnarray}
\mathcal{P}_{k} &=&\left\{ \text{all polynomials on }\mathbb{R}^{3}\text{
with degree at most }k\right\} ;  \label{eq1.9} \\
\overset{\circ }{\mathcal{P}}_{k} &=&\left\{ p\in \mathcal{P}_{k}:\int_{%
\mathbb{S}^{2}}pd\mu =0\right\} ;  \label{eq1.10} \\
H_{k} &=&\left\{ \text{all degree }k\text{ homogeneous polynomials on }%
\mathbb{R}^{3}\right\} ;  \label{eq1.11} \\
\mathcal{H}_{k} &=&\left\{ h\in H_{k}:\Delta _{\mathbb{R}^{3}}h=0\right\} .
\label{eq1.12}
\end{eqnarray}%
It is known that 
\begin{equation}
\left. \mathcal{H}_{k}\right\vert _{\mathbb{S}^{2}}=\left\{ \left.
h\right\vert _{\mathbb{S}^{2}}:h\in \mathcal{H}_{k}\right\}  \label{eq1.13}
\end{equation}%
is exactly the eigenspace of $-\Delta _{\mathbb{S}^{2}}$ associated with
eigenvalue $k\left( k+1\right) $. Moreover%
\begin{equation}
\left. \overset{\circ }{\mathcal{P}}_{k}\right\vert _{\mathbb{S}%
^{2}}=\dbigoplus\limits_{i=1}^{k}\left. \mathcal{H}_{i}\right\vert _{\mathbb{%
S}^{2}}.  \label{eq1.14}
\end{equation}%
We refer the reader to \cite[chapter IV]{SW} for these facts.

\begin{definition}
\label{def1.1}Let $m\in \mathbb{N}$, we denote%
\begin{eqnarray}
&&\mathcal{N}_{m}  \label{eq1.15} \\
&=&\left\{ N\in \mathbb{N}:\exists x_{1},\cdots ,x_{N}\in \mathbb{S}^{2}%
\text{ and }\nu _{1},\cdots ,\nu _{N}\in \left[ 0,\infty \right) \text{ s.t. 
}\nu _{1}+\cdots +\nu _{N}=1\right.  \notag \\
&&\left. \text{and for any }p\in \overset{\circ }{\mathcal{P}}_{m}\text{, }%
\nu _{1}p\left( x_{1}\right) +\cdots +\nu _{N}p\left( x_{N}\right) =0\text{.}%
\right\}  \notag \\
&=&\left\{ N\in \mathbb{N}:\exists x_{1},\cdots ,x_{N}\in \mathbb{S}^{2}%
\text{ and }\nu _{1},\cdots ,\nu _{N}\in \left[ 0,\infty \right) \text{ s.t.
for any }p\in \mathcal{P}_{m}\text{,}\right.  \notag \\
&&\left. \nu _{1}p\left( x_{1}\right) +\cdots +\nu _{N}p\left( x_{N}\right) =%
\frac{1}{4\pi }\int_{\mathbb{S}^{2}}pd\mu \text{.}\right\} \text{.}  \notag
\end{eqnarray}

The smallest number in $\mathcal{N}_{m}$ is denoted as $N_{m}$ i.e. $%
N_{m}=\min \mathcal{N}_{m}$.
\end{definition}

The importance of $N_{m}$ lies in the following theorem, which is the main
result of this paper.

\begin{theorem}
\label{thm1.1}Assume $u\in H^{1}\left( \mathbb{S}^{2}\right) $ such that $%
\int_{\mathbb{S}^{2}}ud\mu =0$ (here $\mu $ is the standard measure on $%
\mathbb{S}^{2}$) and for every $p\in \overset{\circ }{\mathcal{P}}_{m}$, $%
\int_{\mathbb{S}^{2}}pe^{2u}d\mu =0$, then for any $\varepsilon >0$, we have%
\begin{equation}
\log \int_{\mathbb{S}^{2}}e^{2u}d\mu \leq \left( \frac{1}{4\pi N_{m}}%
+\varepsilon \right) \left\Vert \nabla u\right\Vert
_{L^{2}}^{2}+c_{\varepsilon }.  \label{eq1.16}
\end{equation}
\end{theorem}

It is worth pointing out that the coefficient $\frac{1}{4\pi N_{m}}%
+\varepsilon $ is almost optimal (see Lemma \ref{lem3.1}). On the other
hand, in view of (\ref{eq1.6}) and (\ref{eq1.8}), it would be very
interesting to determine the best possible constant $c_{\varepsilon }$ in (%
\ref{eq1.16}) for $m\geq 2$.

The condition in (\ref{eq1.15}) is the same as saying the cubature formula
(a more familiar name of cubature formula is quadrature formula)%
\begin{equation}
\frac{1}{4\pi }\int_{\mathbb{S}^{2}}fd\mu \approx \nu _{1}f\left(
x_{1}\right) +\cdots +\nu _{N}f\left( x_{N}\right)   \label{eq1.17}
\end{equation}%
for functions $f$ on $\mathbb{S}^{2}$ has nonnegative weights and degree of
precision $m$ (here we use the terminology in \cite{HSW}). Various cubature
formulas are of great practical importance in scientific computing and have
been extensively studied in the literature (see the review articles \cite%
{Co, HSW} and the references therein). In particular, the size of $N_{m}$ is
discussed in \cite[section 4.6]{HSW}. It follows from \cite[theorem 7.1]{Co}
or \cite[theorem 4]{HSW} that%
\begin{equation}
N_{m}\geq \left( \left[ \frac{m}{2}\right] +1\right) ^{2}.  \label{eq1.18}
\end{equation}%
Here $\left[ t\right] $ denotes the largest integer less than or equal to $t$%
. In our case when all the weights $\nu _{i}$'s are nonnegative, a simple
proof of (\ref{eq1.18}) is given on \cite[p1203]{HSW}. In general, finding
the exact values of $N_{m}$ for all $m$'s is still an open problem.

On the other hand, it is straightforward to see that $N_{1}=2$ (see Example %
\ref{ex4.1}). Hence (\ref{eq1.4}) follows from Theorem \ref{thm1.1}. It is
also well known in numerical analysis community that $N_{2}=4$ (we provide
an elementary proof of this fact in Lemma \ref{lem4.1} for reader's
convenience). As a consequence, we have

\begin{corollary}
\label{cor1.1}Assume $u\in H^{1}\left( \mathbb{S}^{2}\right) $ such that $%
\int_{\mathbb{S}^{2}}ud\mu =0$ and for every $p\in \overset{\circ }{\mathcal{%
P}}_{2}$, $\int_{\mathbb{S}^{2}}pe^{2u}d\mu =0$, then for any $\varepsilon
>0 $, we have%
\begin{equation}
\log \int_{\mathbb{S}^{2}}e^{2u}d\mu \leq \left( \frac{1}{16\pi }%
+\varepsilon \right) \left\Vert \nabla u\right\Vert
_{L^{2}}^{2}+c_{\varepsilon }.  \label{eq1.19}
\end{equation}
\end{corollary}

At last we want to point out that our analysis of $H^{1}$ on surfaces
depends heavily on the Hilbert space structure of $H^{1}$, and closely
follows \cite[p197]{L}. For similar discussion of $W^{1,n}$ ($n\geq 3$) on a
Riemannian manifold of dimension $n$, \cite[p197]{L} has to use special
symmetrization process to gain the pointwise convergence of the gradient of
functions considered. In \cite{H}, by adapting the approach in this paper,
we are able to avoid the symmetrization process and generalize the analysis
to dimensions at least $3$ as well as higher order Sobolev spaces. We also
remark that in a forthcoming paper \cite{ChG} we discuss an inequality on $%
\mathbb{S}^{2}$ which is the counterpart of the second inequality in the
Szego limit theorem of the Toeplitz determinants on the unit circle.

In Section \ref{sec2}, we will derive some extensions of the concentration
compactness principle in dimension $2$. These refinements will be used in
Section \ref{sec3} to prove our main theorem. In Section \ref{sec4}, we
discuss some elementary facts about $N_{m}$. In particular we will show $%
N_{2}=4$. In Section \ref{sec5}, we will make a first effort toward related
sharp inequalities generalizing (\ref{eq1.6}). In Section \ref{sec6}, we
will show our approach gives a new way to prove the sequence of
Lebedev-Milin type inequalities on the unit circle.

\section{Refinements of concentration compactness principle in dimension $2$%
\label{sec2}}

In this section, we will extend the concentration compactness principle in
dimension $2$ developed in \cite[section I.7]{L}. These extensions will be
crucial in the derivation of Theorem \ref{thm1.1}.

We start from a basic consequence of Moser-Trudinger inequality (\ref{eq1.2}%
).

\begin{lemma}
\label{lem2.1}For any $u\in H^{1}\left( M\right) $ and $a>0$, we have%
\begin{equation}
\int_{M}e^{au^{2}}d\mu <\infty .  \label{eq2.1}
\end{equation}
\end{lemma}

\begin{proof}
Without losing of generality, we can assume $u$ is nonnegative and
unbounded. For $b>0$, let $v=\left( u-b\right) ^{+}$, then%
\begin{equation*}
\left\Vert \nabla v\right\Vert _{L^{2}}^{2}=\int_{u>b}\left\vert \nabla
u\right\vert ^{2}d\mu \rightarrow 0
\end{equation*}%
as $b\rightarrow \infty $. Let $w=v-\overline{v}$, then%
\begin{equation*}
0\leq u\leq v+b=w+\overline{v}+b.
\end{equation*}%
Hence%
\begin{equation*}
u^{2}\leq 2w^{2}+2\left( \overline{v}+b\right) ^{2}.
\end{equation*}%
We have%
\begin{equation*}
e^{au^{2}}\leq e^{2a\left( \overline{v}+b\right) ^{2}}e^{2aw^{2}}\leq
e^{2a\left( \overline{v}+b\right) ^{2}}e^{4\pi \frac{w^{2}}{\left\Vert
\nabla w\right\Vert _{L^{2}}^{2}}}
\end{equation*}%
when $b$ is large enough. It follows that%
\begin{equation*}
\int_{M}e^{au^{2}}d\mu \leq ce^{2a\left( \overline{v}+b\right) ^{2}}<\infty .
\end{equation*}
\end{proof}

Next we prove a localized version of \cite[Theorem I.6]{L}.

\begin{lemma}
\label{lem2.2}Assume $u_{i}\in H^{1}\left( M\right) $ such that $\overline{%
u_{i}}=0$ and $\left\Vert \nabla u_{i}\right\Vert _{L^{2}}\leq 1$. We also
assume $u_{i}\rightharpoonup u$ weakly in $H^{1}\left( M\right) $, $%
u_{i}\rightarrow u$ a.e. and%
\begin{equation}
\left\vert \nabla u_{i}\right\vert ^{2}d\mu \rightarrow \left\vert \nabla
u\right\vert ^{2}d\mu +\sigma  \label{eq2.2}
\end{equation}%
in measure. If $K\subset M$ is a compact subset with $\sigma \left( K\right)
<1$, then for any $1\leq p<\frac{1}{\sigma \left( K\right) }$, we have $%
e^{4\pi u_{i}^{2}}$ is bounded in $L^{p}\left( K\right) $ i.e.%
\begin{equation}
\sup_{i}\int_{K}e^{4\pi pu_{i}^{2}}d\mu <\infty .  \label{eq2.3}
\end{equation}
\end{lemma}

\begin{proof}
For basics about measure theory we refer the readers to \cite{EG}. Let $%
v_{i}=u_{i}-u$, then $v_{i}\rightharpoonup 0$ weakly in $H^{1}\left(
M\right) $, $v_{i}\rightarrow 0$ in $L^{2}\left( M\right) $. For any $%
\varphi \in C^{\infty }\left( M\right) $, we have%
\begin{eqnarray*}
&&\left\Vert \nabla \left( \varphi v_{i}\right) \right\Vert _{L^{2}}^{2} \\
&=&\int_{M}\left( \left\vert \nabla \varphi \right\vert
^{2}v_{i}^{2}+2\varphi v_{i}\nabla \varphi \cdot \nabla v_{i}+\varphi
^{2}\left\vert \nabla v_{i}\right\vert ^{2}\right) d\mu \\
&=&\int_{M}\left\vert \nabla \varphi \right\vert ^{2}v_{i}^{2}d\mu
+2\int_{M}\varphi v_{i}\nabla \varphi \cdot \nabla v_{i}d\mu \\
&&+\int_{M}\left( \varphi ^{2}\left\vert \nabla u_{i}\right\vert
^{2}-2\varphi ^{2}\nabla u\cdot \nabla u_{i}+\varphi ^{2}\left\vert \nabla
u\right\vert ^{2}\right) d\mu \\
&\rightarrow &\int_{M}\varphi ^{2}d\sigma
\end{eqnarray*}%
as $i\rightarrow \infty $. Assume $1\leq p_{1}<\frac{1}{\sigma \left(
K\right) }$, then $\sigma \left( K\right) <\frac{1}{p_{1}}$. Hence there
exists $\varphi \in C^{\infty }\left( M\right) $ such that $\left. \varphi
\right\vert _{K}=1$ and $\int_{M}\varphi ^{2}d\sigma <\frac{1}{p_{1}}$. It
follows that for $i$ large enough,%
\begin{equation*}
\left\Vert \nabla \left( \varphi v_{i}\right) \right\Vert _{L^{2}}^{2}<\frac{%
1}{p_{1}}.
\end{equation*}%
Hence%
\begin{eqnarray*}
\int_{K}e^{4\pi p_{1}\left( v_{i}-\overline{\varphi v_{i}}\right) ^{2}}d\mu
&\leq &\int_{M}e^{4\pi p_{1}\left( \varphi v_{i}-\overline{\varphi v_{i}}%
\right) ^{2}}d\mu \\
&\leq &\int_{M}e^{4\pi \frac{\left( \varphi v_{i}-\overline{\varphi v_{i}}%
\right) ^{2}}{\left\Vert \nabla \left( \varphi v_{i}\right) \right\Vert
_{L^{2}}^{2}}}d\mu \\
&\leq &c\left( M,g\right) .
\end{eqnarray*}%
To continue, we observe that for any $\varepsilon >0$,%
\begin{eqnarray*}
u_{i}^{2} &=&\left( \left( v_{i}-\overline{\varphi v_{i}}\right) +u+%
\overline{\varphi v_{i}}\right) ^{2} \\
&=&\left( v_{i}-\overline{\varphi v_{i}}\right) ^{2}+2\left( v_{i}-\overline{%
\varphi v_{i}}\right) \left( u+\overline{\varphi v_{i}}\right) +\left( u+%
\overline{\varphi v_{i}}\right) ^{2} \\
&\leq &\left( 1+\varepsilon \right) \left( v_{i}-\overline{\varphi v_{i}}%
\right) ^{2}+\left( 1+\varepsilon ^{-1}\right) \left( u+\overline{\varphi
v_{i}}\right) ^{2} \\
&\leq &\left( 1+\varepsilon \right) \left( v_{i}-\overline{\varphi v_{i}}%
\right) ^{2}+2\left( 1+\varepsilon ^{-1}\right) u^{2}+2\left( 1+\varepsilon
^{-1}\right) \overline{\varphi v_{i}}^{2}.
\end{eqnarray*}%
Hence%
\begin{equation*}
e^{4\pi u_{i}^{2}}\leq e^{4\pi \left( 1+\varepsilon \right) \left( v_{i}-%
\overline{\varphi v_{i}}\right) ^{2}}e^{8\pi \left( 1+\varepsilon
^{-1}\right) u^{2}}e^{8\pi \left( 1+\varepsilon ^{-1}\right) \overline{%
\varphi v_{i}}^{2}}.
\end{equation*}%
Given $1\leq p<\frac{1}{\sigma \left( K\right) }$, we can choose a $p_{1}\in
\left( p,\frac{1}{\sigma \left( K\right) }\right) $. There exists a $%
\varepsilon >0$ such that $\frac{p_{1}}{1+\varepsilon }>p$. Note that $%
e^{4\pi \left( 1+\varepsilon \right) \left( v_{i}-\overline{\varphi v_{i}}%
\right) ^{2}}$ is bounded in $L^{\frac{p_{1}}{1+\varepsilon }}\left(
K\right) $, $e^{8\pi \left( 1+\varepsilon ^{-1}\right) u^{2}}\in L^{q}\left(
K\right) $ for any $q<\infty $ (by Lemma \ref{lem2.1}) and $e^{8\pi \left(
1+\varepsilon ^{-1}\right) \overline{\varphi v_{i}}^{2}}\rightarrow 1$ as $%
i\rightarrow \infty $, it follows from Holder's inequality that $e^{4\pi
u_{i}^{2}}$ is bounded in $L^{p}\left( K\right) $.
\end{proof}

\begin{corollary}
\label{cor2.1}With the same assumption as in Lemma \ref{lem2.2}, let%
\begin{equation}
\kappa =\max_{x\in M}\sigma \left( \left\{ x\right\} \right) \leq 1.
\end{equation}

\begin{enumerate}
\item If $\kappa <1$, then for any $1\leq p<\frac{1}{\kappa }$, $e^{4\pi
u_{i}^{2}}$ is bounded in $L^{p}\left( M\right) $. In particular, $e^{4\pi
u_{i}^{2}}\rightarrow e^{4\pi u^{2}}$ in $L^{1}\left( M\right) $.

\item If $\kappa =1$, then $\sigma =\delta _{x_{0}}$ for some $x_{0}\in M$, $%
u=0$ and after passing to a subsequence,%
\begin{equation}
e^{4\pi u_{i}^{2}}\rightarrow 1+c_{0}\delta _{x_{0}}
\end{equation}%
in measure for some $c_{0}\geq 0$.
\end{enumerate}
\end{corollary}

\begin{proof}
First we assume $\kappa <1$. Let $1\leq p<\frac{1}{\kappa }$, then for any $%
x\in M$, $\sigma \left( x\right) <\frac{1}{p}$. Hence for some $r_{x}>0$
small, we have $\sigma \left( \overline{B_{r_{x}}\left( x\right) }\right) <%
\frac{1}{p}$. By the compactness of $M$, we see%
\begin{equation*}
M=\dbigcup\limits_{i=1}^{N}B_{r_{i}}\left( x_{i}\right) .
\end{equation*}%
Here $r_{i}=r_{x_{i}}$. Then%
\begin{equation*}
M=\dbigcup\limits_{i=1}^{N}\overline{B_{r_{i}}\left( x_{i}\right) }.
\end{equation*}%
It follows from the Lemma \ref{lem2.2} that%
\begin{equation*}
\sup_{j}\int_{\overline{B_{r_{i}}\left( x_{i}\right) }}e^{4\pi
pu_{j}^{2}}d\mu <\infty .
\end{equation*}%
Summing up, we get%
\begin{equation*}
\sup_{j}\int_{M}e^{4\pi pu_{j}^{2}}d\mu <\infty .
\end{equation*}%
Next we assume $\kappa =1$. Since%
\begin{equation*}
\int_{M}\left\vert \nabla u\right\vert ^{2}d\mu +\sigma \left( M\right) \leq
1,
\end{equation*}%
and $\overline{u}=0$, we see $u=0$ and $\sigma =\delta _{x_{0}}$ for some $%
x_{0}\in M$. For $r>0$ small, we know $e^{4\pi u_{i}^{2}}$ is bounded in $%
L^{q}\left( M\backslash B_{r}\left( x_{0}\right) \right) $ for any $q<\infty 
$, hence $e^{4\pi u_{i}^{2}}\rightarrow 1$ in $L^{1}\left( M\backslash
B_{r}\left( x_{0}\right) \right) $. It follows that after passing to a
subsequence, $e^{4\pi u_{i}^{2}}\rightarrow 1+c_{0}\delta _{x_{0}}$ in
measure for some $c_{0}\geq 0$.
\end{proof}

Now we are ready to derive the main refinement of the earlier concentration
compactness principle.

\begin{proposition}
\label{prop2.1}Assume $\alpha >0$, $m_{i}>0$, $m_{i}\rightarrow \infty $, $%
u_{i}\in H^{1}\left( M\right) $ such that $\overline{u_{i}}=0$, $\left\Vert
\nabla u_{i}\right\Vert _{L^{2}}=1$ and%
\begin{equation}
\log \int_{M}e^{2m_{i}u_{i}}d\mu \geq \alpha m_{i}^{2}.
\end{equation}%
We also assume $u_{i}\rightharpoonup u$ weakly in $H^{1}\left( M\right) $, $%
\left\vert \nabla u_{i}\right\vert ^{2}d\mu \rightarrow \left\vert \nabla
u\right\vert ^{2}d\mu +\sigma $ in measure and%
\begin{equation}
\frac{e^{2m_{i}u_{i}}}{\int_{M}e^{2m_{i}u_{i}}d\mu }\rightarrow \nu
\end{equation}%
in measure. Let%
\begin{equation}
\left\{ x\in M:\sigma \left( x\right) \geq 4\pi \alpha \right\} =\left\{
x_{1},\cdots ,x_{N}\right\} ,
\end{equation}%
then%
\begin{equation}
\nu =\sum_{i=1}^{N}\nu _{i}\delta _{x_{i}},
\end{equation}%
here $\nu _{i}\geq 0$ and $\sum_{i=1}^{N}\nu _{i}=1$.
\end{proposition}

\begin{proof}
First we claim that if $K$ is a compact subset of $M$ with $\sigma \left(
K\right) <4\pi \alpha $, then $\nu \left( K\right) =0$. Indeed, we can find
another compact set $K_{1}$ such that $K\subset \limfunc{int}K_{1}$ and $%
\sigma \left( K_{1}\right) <4\pi \alpha $. Fix a number $p$ such that%
\begin{equation*}
\frac{1}{4\pi \alpha }<p<\frac{1}{\sigma \left( K_{1}\right) },
\end{equation*}%
then Lemma \ref{lem2.2} tells us%
\begin{equation*}
\int_{K_{1}}e^{4\pi pu_{i}^{2}}d\mu \leq c,
\end{equation*}%
here $c$ is a constant independent of $i$. Using%
\begin{equation*}
2m_{i}u_{i}\leq 4\pi pu_{i}^{2}+\frac{m_{i}^{2}}{4\pi p},
\end{equation*}%
we see%
\begin{equation*}
\int_{K_{1}}e^{2m_{i}u_{i}}d\mu \leq ce^{\frac{m_{i}^{2}}{4\pi p}}.
\end{equation*}%
It follows that%
\begin{equation*}
\frac{\int_{K_{1}}e^{2m_{i}u_{i}}d\mu }{\int_{M}e^{2m_{i}u_{i}}d\mu }\leq
ce^{\left( \frac{1}{4\pi p}-\alpha \right) m_{i}^{2}}.
\end{equation*}%
Hence%
\begin{equation*}
\nu \left( K\right) \leq \nu \left( \limfunc{int}K_{1}\right) \leq \lim
\inf_{i\rightarrow \infty }\frac{\int_{K_{1}}e^{2m_{i}u_{i}}d\mu }{%
\int_{M}e^{2m_{i}u_{i}}d\mu }=0.
\end{equation*}%
It follows that $\nu \left( K\right) =0$.

If $\sigma \left( x\right) <4\pi \alpha $, then for some $r_{x}>0$ small, we
have $\sigma \left( \overline{B_{r_{x}}\left( x\right) }\right) <4\pi \alpha 
$. It follows from the claim that $\nu \left( \overline{B_{r_{x}}\left(
x\right) }\right) =0$. Hence%
\begin{equation*}
\nu \left( M\backslash \left\{ x_{1},\cdots ,x_{N}\right\} \right) =0.
\end{equation*}%
In another word, $\nu =\sum_{i=1}^{N}\nu _{i}\delta _{x_{i}}$ with $\nu
_{i}\geq 0$ and $\sum_{i=1}^{N}\nu _{i}=1$.
\end{proof}

\section{Proof of Theorem \protect\ref{thm1.1}\label{sec3}}

Let $f_{1},\cdots ,f_{L}\in C\left( M\right) $ and $\alpha >0$ be given.
Here is our strategy to show for any $u\in H^{1}\left( M\right) $ with $%
\overline{u}=0$ and $\int_{M}f_{i}e^{2u}d\mu =0$ for $1\leq i\leq L$, we have%
\begin{equation}
\log \int_{M}e^{2u}d\mu \leq \alpha \left\Vert \nabla u\right\Vert
_{L^{2}}^{2}+c.  \label{eq3.1}
\end{equation}%
This will be proven by contradiction argument. If it is not the case, then
there exists $v_{i}\in H^{1}\left( M\right) $, $\overline{v_{i}}=0$, $%
\int_{M}f_{j}e^{2v_{i}}d\mu =0$ for $1\leq j\leq L$, such that%
\begin{equation}
\log \int_{M}e^{2v_{i}}d\mu -\alpha \left\Vert \nabla v_{i}\right\Vert
_{L^{2}}^{2}\rightarrow \infty  \label{eq3.2}
\end{equation}%
as $i\rightarrow \infty $. Then $\log \int_{M}e^{2v_{i}}d\mu \rightarrow
\infty $. Since%
\begin{equation}
\log \int_{M}e^{2v_{i}}d\mu \leq \frac{1}{4\pi }\left\Vert \nabla
v_{i}\right\Vert _{L^{2}}^{2}+c\left( M,g\right) ,  \label{eq3.3}
\end{equation}%
we see $\left\Vert \nabla v_{i}\right\Vert _{L^{2}}\rightarrow \infty $. Let 
$m_{i}=\left\Vert \nabla v_{i}\right\Vert _{L^{2}}$ and $u_{i}=\frac{v_{i}}{%
m_{i}}$, then $m_{i}\rightarrow \infty $, $\left\Vert \nabla
u_{i}\right\Vert _{L^{2}}=1$, $\overline{u_{i}}=0$. After passing to a
subsequence, we have%
\begin{eqnarray*}
u_{i} &\rightharpoonup &u\text{ weakly in }H^{1}\left( M\right) ; \\
\log \int_{M}e^{2m_{i}u_{i}}d\mu -\alpha m_{i}^{2} &\rightarrow &\infty , \\
\left\vert \nabla u_{i}\right\vert ^{2}d\mu &\rightarrow &\left\vert \nabla
u\right\vert ^{2}d\mu +\sigma \text{ in measure,} \\
\frac{e^{2m_{i}u_{i}}}{\int_{M}e^{2m_{i}u_{i}}d\mu } &\rightarrow &\nu \text{
in measure.}
\end{eqnarray*}

Let%
\begin{equation}
\left\{ x\in M:\sigma \left( x\right) \geq 4\pi \alpha \right\} =\left\{
x_{1},\cdots ,x_{N}\right\} ,  \label{eq3.4}
\end{equation}%
then it follows from Proposition \ref{prop2.1} that%
\begin{equation}
\nu =\sum_{i=1}^{N}\nu _{i}\delta _{x_{i}},  \label{eq3.5}
\end{equation}%
here $\nu _{i}\geq 0$ and $\sum_{i=1}^{N}\nu _{i}=1$. On the other hand we
have%
\begin{equation*}
\int_{M}f_{j}d\nu =0
\end{equation*}%
for $1\leq j\leq L$. In another word, we have%
\begin{eqnarray}
4\pi \alpha N &\leq &1;  \label{eq3.6} \\
\sum_{i=1}^{N}\nu _{i}f_{j}\left( x_{i}\right) &=&0  \label{eq3.7}
\end{eqnarray}%
for $1\leq j\leq L$. We hope to get contradiction from these inequalities.

\begin{proof}[Proof of Theorem \protect\ref{thm1.1}]
Let $\alpha =\frac{1}{4\pi N_{m}}+\varepsilon $. If (\ref{eq1.16}) is not
true, then the above discussion gives us $x_{1},\cdots ,x_{N}\in \mathbb{S}%
^{2}$, $\nu _{1},\cdots ,\nu _{N}\geq 0$ such that $\sum_{i=1}^{N}\nu _{i}=1$
and for any $p\in \overset{\circ }{\mathcal{P}}_{m}$, $\nu _{1}p\left(
x_{1}\right) +\cdots +\nu _{N}p\left( x_{N}\right) =0$. Moreover $4\pi
\alpha N\leq 1$. In particular, $N\in \mathcal{N}_{m}$ and hence $N\geq
N_{m} $. It follows that%
\begin{equation*}
\alpha \leq \frac{1}{4\pi N}\leq \frac{1}{4\pi N_{m}}.
\end{equation*}%
This contradicts with the choice of $\alpha $.
\end{proof}

Next we want to show the constant $\frac{1}{4\pi N_{m}}+\varepsilon $ in (%
\ref{eq1.16}) is almost sharp.

\begin{lemma}
\label{lem3.1}Assume $m\in \mathbb{N}$. If $a\geq 0$ and $c\in \mathbb{R}$
such that for any $u\in H^{1}\left( \mathbb{S}^{2}\right) $ with $\overline{u%
}=0$ and $\int_{\mathbb{S}^{2}}pe^{2u}d\mu =0$ for every $p\in \overset{%
\circ }{\mathcal{P}}_{m}$, we have%
\begin{equation}
\log \int_{\mathbb{S}^{2}}e^{2u}d\mu \leq a\left\Vert \nabla u\right\Vert
_{L^{2}}^{2}+c,  \label{eq3.8}
\end{equation}%
then $a\geq \frac{1}{4\pi N_{m}}$.
\end{lemma}

\begin{proof}
First we note that we can rewrite the assumption as for any $u\in
H^{1}\left( \mathbb{S}^{2}\right) $ with $\int_{\mathbb{S}^{2}}pe^{2u}d\mu
=0 $ for every $p\in \overset{\circ }{\mathcal{P}}_{m}$, we have%
\begin{equation}
\log \int_{\mathbb{S}^{2}}e^{2u}d\mu \leq a\left\Vert \nabla u\right\Vert
_{L^{2}}^{2}+2\overline{u}+c.  \label{eq3.9}
\end{equation}

Assume $N\in \mathbb{N}$, $x_{1},\cdots ,x_{N}\in \mathbb{S}^{2}$ and $\nu
_{1},\cdots ,\nu _{N}\in \left[ 0,\infty \right) $ s.t. $\nu _{1}+\cdots
+\nu _{N}=1$ and for any $p\in \overset{\circ }{\mathcal{P}}_{m}$, $\nu
_{1}p\left( x_{1}\right) +\cdots +\nu _{N}p\left( x_{N}\right) =0$. We will
prove $a\geq \frac{1}{4\pi N}$. Lemma \ref{lem3.1} follows. Without losing
of generality we can assume $\nu _{i}>0$ for $1\leq i\leq N$ and $x_{i}\neq
x_{j}$ for $1\leq i<j\leq N$.

To continue let us fix some notations. For $x,y\in \mathbb{S}^{2}$, we
denote $\overline{xy}$ as the geodesic distance between $x$ and $y$ on $%
\mathbb{S}^{2}$. For $r>0$ and $x\in \mathbb{S}^{2}$, we denote $B_{r}\left(
x\right) $ as the geodesic ball with radius $r$ and center $x$ i.e. $%
B_{r}\left( x\right) =\left\{ y\in \mathbb{S}^{2}:\overline{xy}<r\right\} $.

Let $\delta >0$ be small enough such that for $1\leq i<j\leq N$, $\overline{%
B_{2\delta }\left( x_{i}\right) }\cap \overline{B_{2\delta }\left(
x_{j}\right) }=\emptyset $. For $0<\varepsilon <\delta $, we let%
\begin{equation*}
\phi _{\varepsilon }\left( t\right) =\left\{ 
\begin{array}{cc}
2\log \frac{\delta }{\varepsilon }, & 0<t<\varepsilon ; \\ 
2\log \frac{\delta }{t}, & \varepsilon <t<\delta ; \\ 
0, & t>\delta .%
\end{array}%
\right.
\end{equation*}%
If $b\in \mathbb{R}$, then we write%
\begin{equation*}
\phi _{\varepsilon ,b}\left( t\right) =\left\{ 
\begin{array}{cc}
\phi _{\varepsilon }\left( t\right) +b, & 0<t<\delta ; \\ 
b\left( 2-\frac{t}{\delta }\right) , & \delta <t<2\delta ; \\ 
0, & t>2\delta .%
\end{array}%
\right.
\end{equation*}

Let%
\begin{equation}
v\left( x\right) =\sum_{i=1}^{N}\phi _{\varepsilon ,\frac{1}{2}\log \nu
_{i}}\left( \overline{xx_{i}}\right) ,  \label{eq3.10}
\end{equation}%
then%
\begin{eqnarray}
\int_{\mathbb{S}^{2}}e^{2v}d\mu &=&\sum_{i=1}^{N}\int_{B_{\delta }\left(
x_{i}\right) }e^{2\phi _{\varepsilon }\left( \overline{xx_{i}}\right) +\log
\nu _{i}}d\mu +O\left( 1\right)  \label{eq3.11} \\
&=&2\pi \int_{0}^{\delta }e^{2\phi _{\varepsilon }\left( r\right) }\sin
rdr+O\left( 1\right)  \notag \\
&=&2\pi \delta ^{4}\varepsilon ^{-2}+O\left( \log \frac{1}{\varepsilon }%
\right)  \notag
\end{eqnarray}%
as $\varepsilon \rightarrow 0^{+}$.

Note that since $\dim \left( \left. \overset{\circ }{\mathcal{P}}%
_{m}\right\vert _{\mathbb{S}^{2}}\right) =m^{2}+2m$, we can fix $%
p_{1},\cdots ,p_{m^{2}+2m}\in \overset{\circ }{\mathcal{P}}_{m}$ such that $%
\left. p_{1}\right\vert _{\mathbb{S}^{2}},\cdots ,\left.
p_{m^{2}+2m}\right\vert _{\mathbb{S}^{2}}$ is a base for $\left. \overset{%
\circ }{\mathcal{P}}_{m}\right\vert _{\mathbb{S}^{2}}$. For $1\leq j\leq
m^{2}+2m$, we have%
\begin{equation}
\int_{\mathbb{S}^{2}}e^{2v}p_{j}d\mu =O\left( \log \frac{1}{\varepsilon }%
\right)  \label{eq3.12}
\end{equation}%
as $\varepsilon \rightarrow 0^{+}$. Indeed,%
\begin{eqnarray*}
&&\int_{\mathbb{S}^{2}}e^{2v}p_{j}d\mu \\
&=&\dsum\limits_{i=1}^{N}\nu _{i}\int_{B_{\delta }\left( x_{i}\right)
}e^{\phi _{\varepsilon }\left( \overline{xx_{i}}\right) }p_{j}\left(
x\right) d\mu \left( x\right) +O\left( 1\right) \\
&=&\dsum\limits_{i=1}^{N}\left( \nu _{i}p_{j}\left( x_{i}\right)
\int_{B_{\delta }\left( x_{i}\right) }e^{\phi _{\varepsilon }\left( 
\overline{xx_{i}}\right) }d\mu \left( x\right) +\int_{B_{\delta }\left(
x_{i}\right) }e^{\phi _{\varepsilon }\left( \overline{xx_{i}}\right)
}O\left( \overline{xx_{i}}^{2}\right) d\mu \left( x\right) \right) +O\left(
1\right) ,
\end{eqnarray*}%
here we have used the Talyor expansion of $p_{j}$ near $x_{i}$ and the
vanishing of integral of first order terms by symmetry. Using%
\begin{equation*}
\dsum\limits_{i=1}^{N}\nu _{i}p_{j}\left( x_{i}\right) =0,
\end{equation*}%
we see%
\begin{equation*}
\int_{\mathbb{S}^{2}}e^{2v}p_{j}d\mu =O\left( \log \frac{1}{\varepsilon }%
\right) .
\end{equation*}

To get a test function satisfying orthogornality condition, we need to do
some corrections. We first claim that there exists $\psi _{1},\cdots ,\psi
_{m^{2}+2m}\in C_{c}^{\infty }\left( \mathbb{S}^{2}\backslash
\dbigcup\limits_{i=1}^{N}\overline{B_{2\delta }\left( x_{i}\right) }\right) $
such that the determinant%
\begin{equation}
\det \left[ \int_{\mathbb{S}^{2}}\psi _{j}p_{k}d\mu \right] _{1\leq j,k\leq
m^{2}+2m}\neq 0.  \label{eq3.13}
\end{equation}%
Indeed, here is one way to construct these functions. Fix a nonzero smooth
function $\eta \in C_{c}^{\infty }\left( \mathbb{S}^{2}\backslash
\dbigcup\limits_{i=1}^{N}\overline{B_{2\delta }\left( x_{i}\right) }\right) $%
, then $\eta p_{1},\cdots ,\eta p_{m^{2}+2m}$ are linearly independent. It
follows that the matrix%
\begin{equation*}
\left[ \int_{\mathbb{S}^{2}}\eta ^{2}p_{j}p_{k}d\mu \right] _{1\leq j,k\leq
m^{2}+2m}
\end{equation*}%
is positive definite and has positive determinant. Then $\psi _{j}=\eta
^{2}p_{j}$ satisfies the claim.

It follows from (\ref{eq3.13}) that we can find $\beta _{1},\cdots ,\beta
_{m^{2}+2m}\in \mathbb{R}$ such that%
\begin{equation}
\int_{\mathbb{S}^{2}}\left( e^{2v}+\sum_{j=1}^{m^{2}+2m}\beta _{j}\psi
_{j}\right) p_{k}d\mu =0  \label{eq3.14}
\end{equation}%
for $k=1,\cdots ,m^{2}+2m$. Moreover%
\begin{equation}
\beta _{j}=O\left( \log \frac{1}{\varepsilon }\right)  \label{eq3.15}
\end{equation}%
as $\varepsilon \rightarrow 0^{+}$. As a consequence we can find a constant $%
c_{1}>0$ such that%
\begin{equation}
\sum_{j=1}^{m^{2}+2m}\beta _{j}\psi _{j}+c_{1}\log \frac{1}{\varepsilon }%
\geq \log \frac{1}{\varepsilon }.  \label{eq3.16}
\end{equation}%
We define $u$ as%
\begin{equation}
e^{2u}=e^{2v}+\sum_{j=1}^{m^{2}+2m}\beta _{j}\psi _{j}+c_{1}\log \frac{1}{%
\varepsilon }.  \label{eq3.17}
\end{equation}%
Note this $u$ will be the test function we use to prove Lemma \ref{lem3.1}.

It follows from (\ref{eq3.14}) that $\int_{\mathbb{S}^{2}}e^{2u}pd\mu =0$
for all $p\in \overset{\circ }{\mathcal{P}}_{m}$. Moreover using (\ref%
{eq3.11}) and (\ref{eq3.15}) we see%
\begin{equation}
\int_{\mathbb{S}^{2}}e^{2u}d\mu =2\pi \delta ^{4}\varepsilon ^{-2}+O\left(
\log \frac{1}{\varepsilon }\right) =2\pi \delta ^{4}\varepsilon ^{-2}\left(
1+o\left( 1\right) \right) ,  \label{eq3.18}
\end{equation}%
hence%
\begin{equation}
\log \int_{\mathbb{S}^{2}}e^{2u}d\mu =2\log \frac{1}{\varepsilon }+O\left(
1\right)  \label{eq3.19}
\end{equation}%
as $\varepsilon \rightarrow 0^{+}$. Calculation shows%
\begin{equation}
\overline{u}=o\left( \log \frac{1}{\varepsilon }\right) .  \label{eq3.20}
\end{equation}%
At last we claim%
\begin{equation}
\int_{\mathbb{S}^{2}}\left\vert \nabla u\right\vert ^{2}d\mu =8\pi N\log 
\frac{1}{\varepsilon }+o\left( \log \frac{1}{\varepsilon }\right) .
\label{eq3.21}
\end{equation}%
Once this is known, we plug $u$ into (\ref{eq3.9}) and get%
\begin{equation*}
2\log \frac{1}{\varepsilon }\leq 8\pi Na\log \frac{1}{\varepsilon }+o\left(
\log \frac{1}{\varepsilon }\right) .
\end{equation*}%
Divide $\log \frac{1}{\varepsilon }$ on both sides and let $\varepsilon
\rightarrow 0^{+}$, we see $a\geq \frac{1}{4\pi N}$.

To derive (\ref{eq3.21}), we note that on $\mathbb{S}^{2}\backslash
\dbigcup\limits_{i=1}^{N}\overline{B_{2\delta }\left( x_{i}\right) }$, $%
\left\vert \nabla u\right\vert =O\left( 1\right) $ (here we need to use (\ref%
{eq3.15}) and (\ref{eq3.16})), hence%
\begin{eqnarray*}
\int_{\mathbb{S}^{2}}\left\vert \nabla u\right\vert ^{2}d\mu
&=&\sum_{i=1}^{N}\int_{B_{2\delta }\left( x_{i}\right) }\left\vert \nabla
u\right\vert ^{2}d\mu +O\left( 1\right) \\
&=&\sum_{i=1}^{N}\int_{B_{\delta }\left( x_{i}\right) }\left\vert \nabla
u\right\vert ^{2}d\mu +O\left( 1\right) \\
&=&\sum_{i=1}^{N}8\pi \int_{\varepsilon }^{\delta }\frac{r^{-10}\sin r}{%
\left( \frac{c_{1}\log \frac{1}{\varepsilon }}{\nu _{i}\delta ^{4}}%
+r^{-4}\right) ^{2}}dr+O\left( 1\right) \\
&=&8\pi N\log \frac{1}{\varepsilon }+o\left( \log \frac{1}{\varepsilon }%
\right) .
\end{eqnarray*}
\end{proof}

\section{The number $N_{m}$\label{sec4}}

We start with the following basic observation.

\begin{example}
\label{ex4.1}$N_{1}=2$. It is clear that $N_{1}\geq 2$, on the other hand,
by setting $\nu _{1}=\nu _{2}=\frac{1}{2}$ and $x_{2}=-x_{1}$, we see $%
N_{1}\leq 2$. Hence $N_{1}=2$.
\end{example}

\begin{lemma}
\label{lem4.1}$N_{2}=4$.
\end{lemma}

\begin{proof}
Indeed it follows from (\ref{eq1.18}) that $N_{2}\geq 4$. Here we give a
direct proof. Note that $N_{2}\geq N_{1}=2$.

If $N_{2}=2$, then we have $\nu _{1}x_{1}+\nu _{2}x_{2}=0$. It implies $\nu
_{1}=\nu _{2}=\frac{1}{2}$. Hence $x_{2}=-x_{1}$. By rotation, we assume $%
x_{1}=\left( 0,0,1\right) $. Let $p\left( y\right) =y_{1}^{2}$, then%
\begin{equation*}
\nu _{1}p\left( x_{1}\right) +\nu _{2}p\left( x_{2}\right) =0\neq \frac{1}{%
4\pi }\int_{\mathbb{S}^{2}}pd\mu .
\end{equation*}%
We get a contradiction.

If $N_{2}=3$, then we have $\nu _{1}x_{1}+\nu _{2}x_{2}+\nu _{3}x_{3}=0$. It
follows that $x_{1},x_{2},x_{3}$ must lie in a plane. By rotation we can
assume that plane is the horizontal plane. Let $p=y_{3}^{2}$, then%
\begin{equation*}
\nu _{1}p\left( x_{1}\right) +\nu _{2}p\left( x_{2}\right) +\nu _{3}p\left(
x_{3}\right) =0\neq \frac{1}{4\pi }\int_{\mathbb{S}^{2}}pd\mu .
\end{equation*}%
This gives us a contradiction.

Hence we only need to find $x_{1},x_{2},x_{3},x_{4}\in \mathbb{S}^{2}$, $\nu
_{1},\nu _{2},\nu _{3},\nu _{4}\geq 0$ with $\nu _{1}+\nu _{2}+\nu _{3}+\nu
_{4}=1$ such that for any $p\in \overset{\circ }{\mathcal{P}}_{2}$, we have%
\begin{equation}
\nu _{1}p\left( x_{1}\right) +\nu _{2}p\left( x_{2}\right) +\nu _{3}p\left(
x_{3}\right) +\nu _{4}p\left( x_{4}\right) =0.  \label{eq4.1}
\end{equation}%
We claim the four vortices of a regular tetrahedron inside the unit sphere
with $\nu _{i}=\frac{1}{4}$ for $1\leq i\leq 4$ would satisfy the property.
Indeed, let%
\begin{eqnarray*}
x_{1} &=&\left( 0,0,1\right) ; \\
x_{2} &=&\left( 0,\frac{2\sqrt{2}}{3},-\frac{1}{3}\right) ; \\
x_{3} &=&\left( \sqrt{\frac{2}{3}},-\frac{\sqrt{2}}{3},-\frac{1}{3}\right) ;
\\
x_{4} &=&\left( -\sqrt{\frac{2}{3}},-\frac{\sqrt{2}}{3},-\frac{1}{3}\right) .
\end{eqnarray*}%
Then we have%
\begin{equation*}
x_{1}+x_{2}+x_{3}+x_{4}=0.
\end{equation*}%
Moreover using%
\begin{equation*}
\mathcal{H}_{2}=\limfunc{span}\left\{ y_{1}^{2}-\frac{\left\vert
y\right\vert ^{2}}{3},y_{2}^{2}-\frac{\left\vert y\right\vert ^{2}}{3}%
,y_{1}y_{2},y_{1}y_{3},y_{2}y_{3}\right\} ,
\end{equation*}%
checking (\ref{eq4.1}) for each $p$ in the base verifies the identity.
\end{proof}

It remains an interesting question to find $N_{m}$ for all $m$'s.

\section{A sharp inequality by perturbation\label{sec5}}

In this section we prove a sharp inequality by the perturbation method in
the same spirit as \cite{ChY1}.

\begin{theorem}
\label{thm5.1}There exists an $a_{0}<\frac{1}{8\pi }$ such that for all $%
u\in H^{1}\left( \mathbb{S}^{2}\right) $ satisfying $\int_{\mathbb{S}%
^{2}}ud\mu =0$ and for every $p\in \overset{\circ }{\mathcal{P}}_{2}$, $%
\int_{\mathbb{S}^{2}}pe^{2u}d\mu =0$, we have%
\begin{equation}
\log \left( \frac{1}{4\pi }\int_{\mathbb{S}^{2}}e^{2u}d\mu \right) \leq
a_{0}\left\Vert \nabla u\right\Vert _{L^{2}}^{2}.  \label{eq5.1}
\end{equation}
\end{theorem}

For convenience we denote%
\begin{equation}
\mathcal{S}_{2}=\left\{ u\in H^{1}\left( \mathbb{S}^{2}\right) :\overline{u}%
=0,\int_{\mathbb{S}^{2}}pe^{2u}d\mu =0\text{ for all }p\in \overset{\circ }{%
\mathcal{P}}_{2}\right\} .  \label{eq5.2}
\end{equation}%
For a given number $a\in \left( \frac{1}{16\pi },\frac{1}{8\pi }\right) $,
it follows from Corollary \ref{cor1.1} that for every $u\in \mathcal{S}_{2}$,%
\begin{equation}
\log \left( \frac{1}{4\pi }\int_{\mathbb{S}^{2}}e^{2u}d\mu \right) \leq
a\left\Vert \nabla u\right\Vert _{L^{2}}^{2}+c_{a}.  \label{eq5.3}
\end{equation}%
Let%
\begin{equation}
s=s_{a}=\inf_{u\in \mathcal{S}_{2}}\left[ a\left\Vert \nabla u\right\Vert
_{L^{2}}^{2}-\log \left( \frac{1}{4\pi }\int_{\mathbb{S}^{2}}e^{2u}d\mu
\right) \right] .  \label{eq5.4}
\end{equation}%
We claim $s$ is achieved. Indeed if $u_{i}\in \mathcal{S}_{2}$ is a
minimizing sequence, then%
\begin{equation*}
a\left\Vert \nabla u_{i}\right\Vert _{L^{2}}^{2}-\log \left( \frac{1}{4\pi }%
\int_{\mathbb{S}^{2}}e^{2u_{i}}d\mu \right) \leq c.
\end{equation*}%
Here $c$ is a constant independent of $i$. Choose a number $\varepsilon $
with $0<\varepsilon <a-\frac{1}{16\pi }$. Using Corollary \ref{cor1.1} we
have%
\begin{equation*}
a\left\Vert \nabla u_{i}\right\Vert _{L^{2}}^{2}\leq \log \left( \frac{1}{%
4\pi }\int_{S^{2}}e^{2u_{i}}d\mu \right) +c\leq \left( \frac{1}{16\pi }%
+\varepsilon \right) \left\Vert \nabla u_{i}\right\Vert _{L^{2}}^{2}+c.
\end{equation*}%
It follows that%
\begin{equation*}
\left\Vert \nabla u_{i}\right\Vert _{L^{2}}\leq c.
\end{equation*}%
After passing to a subsequence we can find $u\in H^{1}\left( \mathbb{S}%
^{2}\right) $ such that $u_{i}\rightharpoonup u$ weakly in $H^{1}\left( 
\mathbb{S}^{2}\right) $. Hence $u_{i}\rightarrow u$ in $L^{2}\left( \mathbb{S%
}^{2}\right) $ and we can also assume $u_{i}\rightarrow u$ a.e. For any $b>0$%
, we have%
\begin{equation*}
2bu_{i}\leq 4\pi \frac{u_{i}^{2}}{\left\Vert \nabla u_{i}\right\Vert
_{L^{2}}^{2}}+\frac{b^{2}\left\Vert \nabla u_{i}\right\Vert _{L^{2}}^{2}}{%
4\pi }.
\end{equation*}%
Hence%
\begin{equation*}
\int_{\mathbb{S}^{2}}e^{2bu_{i}}d\mu \leq ce^{\frac{b^{2}\left\Vert \nabla
u_{i}\right\Vert _{L^{2}}^{2}}{4\pi }}\leq c.
\end{equation*}%
It follows that $e^{2u_{i}}\rightarrow e^{2u}$ in $L^{1}\left( \mathbb{S}%
^{2}\right) $. Hence for any $p\in \overset{\circ }{\mathcal{P}}_{2}$, $%
\int_{\mathbb{S}^{2}}pe^{2u}d\mu =0$. It follows that $u\in \mathcal{S}_{2}$.%
\begin{eqnarray*}
s &\leq &a\left\Vert \nabla u\right\Vert _{L^{2}}^{2}-\log \left( \frac{1}{%
4\pi }\int_{\mathbb{S}^{2}}e^{2u}d\mu \right) \\
&\leq &\lim \inf_{i\rightarrow \infty }\left[ a\left\Vert \nabla
u_{i}\right\Vert _{L^{2}}^{2}-\log \left( \frac{1}{4\pi }\int_{\mathbb{S}%
^{2}}e^{2u_{i}}d\mu \right) \right] \\
&=&s.
\end{eqnarray*}%
Hence $u$ is a minimizer.

Let $u_{a}$ be a minimizer for (\ref{eq5.4}). When no confusion would
happen, we simply write $u$ instead of $u_{a}$. We will show that if $a$ is
close enough to $\frac{1}{8\pi }$, the minimizer $u$ must be identically
zero. This would imply Theorem \ref{thm5.1}.

To achieve this aim, we can assume $\frac{5}{48\pi }<a<\frac{1}{8\pi }$.
Since $u$ is a minimizer, we see%
\begin{equation*}
a\left\Vert \nabla u\right\Vert _{L^{2}}^{2}-\log \left( \frac{1}{4\pi }%
\int_{\mathbb{S}^{2}}e^{2u}d\mu \right) \leq 0.
\end{equation*}%
Hence applying Corollary \ref{cor1.1} we get%
\begin{equation*}
a\left\Vert \nabla u\right\Vert _{L^{2}}^{2}\leq \log \left( \frac{1}{4\pi }%
\int_{\mathbb{S}^{2}}e^{2u}d\mu \right) \leq \frac{1}{12\pi }\left\Vert
\nabla u\right\Vert _{L^{2}}^{2}+c.
\end{equation*}%
It implies $\left\Vert \nabla u\right\Vert _{L^{2}}^{2}\leq c$, a constant
independent of $a$.

Next we claim that as $a\rightarrow \frac{1}{8\pi }$, $u_{a}\rightharpoonup
0 $ weakly in $H^{1}\left( \mathbb{S}^{2}\right) $. Indeed if this is not
the case, then we can find a sequence $a_{i}\rightarrow \frac{1}{8\pi }$, $%
u_{i}=u_{a_{i}}$ such that $u_{i}\rightharpoonup w$ weakly in $H^{1}\left( 
\mathbb{S}^{2}\right) $ and $w\neq 0$. We can also assume $u_{i}\rightarrow
w $ a.e. It follows from classical Moser-Trudinger inequality (see (\ref%
{eq1.2})) that $e^{2u_{i}}\rightarrow e^{2w}$ in $L^{1}\left( \mathbb{S}%
^{2}\right) $. Hence $w\in \mathcal{S}_{2}$. Since%
\begin{equation*}
a_{i}\left\Vert \nabla u_{i}\right\Vert _{L^{2}}^{2}\leq \log \left( \frac{1%
}{4\pi }\int_{\mathbb{S}^{2}}e^{2u_{i}}d\mu \right) ,
\end{equation*}%
taking a limit we get%
\begin{equation*}
\frac{1}{8\pi }\left\Vert \nabla w\right\Vert _{L^{2}}^{2}\leq \log \left( 
\frac{1}{4\pi }\int_{\mathbb{S}^{2}}e^{2w}d\mu \right) .
\end{equation*}%
It follows from equality case of (\ref{eq1.4}) (see \cite{GuM}) that $w=0$.
This gives us a contradiction.

Applying the Moser-Trudinger inequality (\ref{eq1.2}) again we see for any $%
b>0$, $e^{2bu_{a}}\rightarrow 1$ in $L^{q}\left( S^{2}\right) $ for any $%
q\in \left[ 1,\infty \right) $ as $a\rightarrow \frac{1}{8\pi }$. Hence%
\begin{equation*}
a\left\Vert \nabla u_{a}\right\Vert _{L^{2}}^{2}\leq \log \left( \frac{1}{%
4\pi }\int_{\mathbb{S}^{2}}e^{2u_{a}}d\mu \right) \rightarrow 0.
\end{equation*}%
It follows that $\left\Vert \nabla u_{a}\right\Vert _{L^{2}}=o\left(
1\right) $ as $a\rightarrow \frac{1}{8\pi }$.

To continue we observe that since%
\begin{equation*}
\left. \overset{\circ }{\mathcal{P}}_{2}\right\vert _{\mathbb{S}^{2}}=\left. 
\mathcal{H}_{1}\right\vert _{\mathbb{S}^{2}}\oplus \left. \mathcal{H}%
_{2}\right\vert _{\mathbb{S}^{2}}=\left. \left( \mathcal{H}_{1}+\mathcal{H}%
_{2}\right) \right\vert _{\mathbb{S}^{2}},
\end{equation*}%
$u$ satisfies the Euler-Langrage equation%
\begin{equation}
-a\Delta u-\frac{e^{2u}}{\int_{\mathbb{S}^{2}}e^{2u}d\mu }=-\frac{1}{4\pi }%
+\ell e^{2u}+he^{2u}  \label{eq5.5}
\end{equation}%
for some $\ell =\ell _{a}\in \mathcal{H}_{1}$ and $h=h_{a}\in \mathcal{H}%
_{2} $.

Since $\mathcal{H}_{1}+\mathcal{H}_{2}$ is a finite dimensional vector
space, any two norms on it are equivalent. Hence we fix an arbitrary norm on 
$\mathcal{H}_{1}+\mathcal{H}_{2}$ from now on. We claim that $\ell
_{a}\rightarrow 0$ and $h_{a}\rightarrow 0$ as $a\rightarrow \frac{1}{8\pi }$%
. For convenience we write%
\begin{equation*}
\lambda =\frac{1}{4\pi }\int_{\mathbb{S}^{2}}e^{2u}d\mu .
\end{equation*}%
Note that $\lambda =1+o\left( 1\right) $. The equation becomes%
\begin{equation}
-a\Delta u+\frac{1}{4\pi }=e^{2u}\left( \frac{1}{4\pi \lambda }+\ell
+h\right) .  \label{eq5.6}
\end{equation}%
Multiplying $\frac{1}{4\pi \lambda }+\ell +h$ and integrating on $\mathbb{S}%
^{2}$, we see%
\begin{equation*}
\int_{\mathbb{S}^{2}}\left( -a\Delta u+\frac{1}{4\pi }\right) \left( \frac{1%
}{4\pi \lambda }+\ell +h\right) d\mu =\int_{\mathbb{S}^{2}}e^{2u}\left( 
\frac{1}{4\pi \lambda }+\ell +h\right) ^{2}d\mu .
\end{equation*}%
Using the fact $u\in \mathcal{S}_{2}$ it becomes%
\begin{eqnarray*}
&&a\int_{\mathbb{S}^{2}}u\left( 2\ell +6h\right) d\mu \\
&=&\int_{\mathbb{S}^{2}}e^{2u}\left( \ell +h\right) ^{2}d\mu \\
&=&\int_{\mathbb{S}^{2}}\left( e^{2u}-1\right) \left( \ell +h\right)
^{2}d\mu +\int_{\mathbb{S}^{2}}\ell ^{2}d\mu +\int_{\mathbb{S}^{2}}h^{2}d\mu
.
\end{eqnarray*}%
It follows that%
\begin{equation*}
o\left( \left\Vert \ell \right\Vert +\left\Vert h\right\Vert \right) =\int_{%
\mathbb{S}^{2}}\ell ^{2}d\mu +\int_{\mathbb{S}^{2}}h^{2}d\mu +o\left(
\left\Vert \ell \right\Vert ^{2}+\left\Vert h\right\Vert ^{2}\right) .
\end{equation*}%
Hence%
\begin{equation*}
\left\Vert \ell \right\Vert ^{2}+\left\Vert h\right\Vert ^{2}=o\left(
\left\Vert \ell \right\Vert +\left\Vert h\right\Vert \right) .
\end{equation*}%
We get $\left\Vert \ell \right\Vert +\left\Vert h\right\Vert =o\left(
1\right) $.

Now we claim that $\left\Vert u_{a}\right\Vert _{L^{\infty }}=o\left(
1\right) $. Indeed since%
\begin{eqnarray*}
&&\left\Vert e^{2u}\left( \frac{1}{4\pi \lambda }+\ell +h\right) -\frac{1}{%
4\pi }\right\Vert _{L^{2}} \\
&\leq &\left\Vert e^{2u}\left( \frac{1}{4\pi \lambda }-\frac{1}{4\pi }%
\right) \right\Vert _{L^{2}}+\frac{1}{4\pi }\left\Vert e^{2u}-1\right\Vert
_{L^{2}}+\left\Vert e^{2u}\left( \ell +h\right) \right\Vert _{L^{2}} \\
&=&o\left( 1\right) ,
\end{eqnarray*}%
it follows from (\ref{eq5.6}) and standard elliptic theory that $\left\Vert
u_{a}\right\Vert _{W^{2,2}}=o\left( 1\right) $. Sobolev embedding theorem
tells us $\left\Vert u_{a}\right\Vert _{L^{\infty }}=o\left( 1\right) $.

At last we observe that $e^{2u}-\lambda $ is perpendicular to $\mathbb{R}$, $%
\mathcal{H}_{1}$ and $\mathcal{H}_{2}$, hence%
\begin{eqnarray*}
&&12\int_{\mathbb{S}^{2}}\left( e^{2u}-\lambda \right) ^{2}d\mu \\
&\leq &\int_{\mathbb{S}^{2}}\left\vert \nabla e^{2u}\right\vert ^{2}d\mu \\
&=&4\int_{\mathbb{S}^{2}}e^{4u}\left\vert \nabla u\right\vert ^{2}d\mu \\
&=&\int_{\mathbb{S}^{2}}\nabla u\cdot \nabla e^{4u}d\mu \\
&=&\int_{\mathbb{S}^{2}}\left( -\Delta u\right) e^{4u}d\mu \\
&=&\int_{\mathbb{S}^{2}}\left( -\Delta u\right) \left( e^{4u}-\lambda
^{2}\right) d\mu \\
&=&\frac{1}{a}\int_{\mathbb{S}^{2}}\left[ e^{2u}\left( \frac{1}{4\pi \lambda 
}+\ell +h\right) -\frac{1}{4\pi }\right] \left( e^{4u}-\lambda ^{2}\right)
d\mu \\
&=&\frac{1+o\left( 1\right) }{2\pi a}\int_{\mathbb{S}^{2}}\left(
e^{2u}-\lambda \right) ^{2}d\mu +\frac{1}{a}\int_{\mathbb{S}%
^{2}}e^{2u}\left( \ell +h\right) \left( e^{4u}-\lambda ^{2}\right) d\mu .
\end{eqnarray*}%
On the other hand,%
\begin{eqnarray*}
&&\int_{\mathbb{S}^{2}}e^{2u}\left( \ell +h\right) \left( e^{4u}-\lambda
^{2}\right) d\mu \\
&=&\int_{\mathbb{S}^{2}}\left( e^{2u}-\lambda \right) \left( \ell +h\right)
\left( e^{4u}-\lambda ^{2}\right) d\mu +\lambda \int_{\mathbb{S}^{2}}\left(
\ell +h\right) \left( e^{4u}-\lambda ^{2}\right) d\mu \\
&=&o\left( 1\right) \int_{\mathbb{S}^{2}}\left( e^{2u}-\lambda \right)
^{2}d\mu +\lambda \int_{\mathbb{S}^{2}}\left( \ell +h\right) \left(
e^{4u}-2\lambda e^{2u}+\lambda ^{2}\right) d\mu \\
&=&o\left( 1\right) \int_{\mathbb{S}^{2}}\left( e^{2u}-\lambda \right)
^{2}d\mu +\lambda \int_{\mathbb{S}^{2}}\left( \ell +h\right) \left(
e^{2u}-\lambda \right) ^{2}d\mu \\
&=&o\left( 1\right) \int_{\mathbb{S}^{2}}\left( e^{2u}-\lambda \right)
^{2}d\mu .
\end{eqnarray*}%
Here we have used the fact $u\in \mathcal{S}_{2}$. Plug this equality back
we see%
\begin{equation*}
\left( 12-\frac{1}{2\pi a}+o\left( 1\right) \right) \int_{\mathbb{S}%
^{2}}\left( e^{2u}-\lambda \right) ^{2}d\mu \leq 0.
\end{equation*}%
Since $a$ is close to $\frac{1}{8\pi }$, we get $\int_{\mathbb{S}^{2}}\left(
e^{2u}-\lambda \right) ^{2}d\mu =0$. Hence $u$ must be constant function. In
view of the fact $\overline{u}=0$, we get $u=0$. This finishes the proof of
Theorem \ref{thm5.1}.

\section{A revisit of Lebedev-Milin type inequalities on $\mathbb{S}^{1}$%
\label{sec6}}

In this section we will show the above method on $\mathbb{S}^{2}$ provides a
variational approach for a sequence of Lebedev-Milin type inequalities on $%
\mathbb{S}^{1}$. Let $D$ be the unit disk in the plane and $\mathbb{S}%
^{1}=\partial D$ be the unit circle. We use $\theta $ as the usual angle
variable and identify $\mathbb{R}^{2}$ as $\mathbb{C}$.

\begin{theorem}
\label{thm6.1}For $m\in \mathbb{N}$, $u\in H^{1}\left( D\right) $ with $%
\int_{\mathbb{S}^{1}}ud\theta =0$ and $\int_{\mathbb{S}^{1}}e^{u}e^{ik\theta
}d\theta =0$ for $k=1,\cdots ,m$, we have%
\begin{equation}
\log \left( \frac{1}{2\pi }\int_{\mathbb{S}^{1}}e^{u}d\theta \right) \leq 
\frac{1}{4\pi \left( m+1\right) }\left\Vert \nabla u\right\Vert
_{L^{2}\left( D\right) }^{2}.  \label{eq6.1}
\end{equation}%
Moreover equality holds if and only if $u\left( z\right) =\log \frac{1}{%
\left\vert 1-\xi z^{m+1}\right\vert ^{2}}$ for some $\xi \in \mathbb{C}$
with $\left\vert \xi \right\vert <1$.
\end{theorem}

For $m=1$, (\ref{eq6.1}) is proved in \cite{OsPS} by variational method. As
observed in \cite{Wi}, (\ref{eq6.1}) follows from the work of
Grenander-Szego \cite{GrS} on Toeplitz determinants.

On $\mathbb{S}^{1}$, the Moser-Trudinger inequality (\ref{eq1.2}) is
replaced by the Beurling-Chang-Marshall inequality (see \cite[corollary 2]%
{ChM}): for $u\in H^{1}\left( D\right) \backslash \left\{ 0\right\} $ with $%
\int_{\mathbb{S}^{1}}ud\theta =0$, we have%
\begin{equation}
\int_{\mathbb{S}^{1}}e^{\pi \frac{u^{2}}{\left\Vert \nabla u\right\Vert
_{L^{2}\left( D\right) }^{2}}}d\theta \leq c.  \label{eq6.2}
\end{equation}

Similar to (\ref{eq1.9})--(\ref{eq1.12}), for any nonnegative integer $k$,
we write%
\begin{eqnarray}
\mathcal{P}_{k} &=&\left\{ \text{real polynomials on }\mathbb{R}^{2}\text{
with degree at most }k\right\} ;  \label{eq6.3} \\
\overset{\circ }{\mathcal{P}}_{k} &=&\left\{ p\in \mathcal{P}_{k}:\int_{%
\mathbb{S}^{1}}pd\theta =0\right\} ;  \label{eq6.4} \\
H_{k} &=&\left\{ \text{degree }k\text{ homogeneous real polynomials on }%
\mathbb{R}^{2}\right\} ;  \label{eq6.5} \\
\mathcal{H}_{k} &=&\left\{ h\in H_{k}:\Delta _{\mathbb{R}^{2}}h=0\right\}
=span_{\mathbb{R}}\left\{ \func{Re}\left( z^{k}\right) ,\func{Im}\left(
z^{k}\right) \right\} .  \label{eq6.6}
\end{eqnarray}%
Note that%
\begin{equation}
\left. \mathcal{H}_{k}\right\vert _{\mathbb{S}^{1}}=span_{\mathbb{R}}\left\{
\cos k\theta ,\sin k\theta \right\}  \label{eq6.7}
\end{equation}%
and%
\begin{equation}
\left. \overset{\circ }{\mathcal{P}}_{k}\right\vert _{\mathbb{S}^{1}}=span_{%
\mathbb{R}}\left\{ \cos j\theta ,\sin j\theta :j\in \mathbb{N},j\leq
k\right\} .  \label{eq6.8}
\end{equation}

Corresponds to Definition \ref{def1.1}, we have for $m\in \mathbb{N}$,%
\begin{eqnarray}
&&\mathcal{N}_{m}\left( \mathbb{S}^{1}\right)  \label{eq6.9} \\
&=&\left\{ N\in \mathbb{N}:\exists z_{1},\cdots ,z_{N}\in \mathbb{S}^{1}%
\text{ and }\nu _{1},\cdots ,\nu _{N}\in \left[ 0,\infty \right) \text{ s.t.
for any }p\in \mathcal{P}_{m}\text{,}\right.  \notag \\
&&\left. \nu _{1}p\left( z_{1}\right) +\cdots +\nu _{N}p\left( z_{N}\right) =%
\frac{1}{2\pi }\int_{\mathbb{S}^{1}}pd\theta \text{.}\right\}  \notag
\end{eqnarray}%
and $N_{m}\left( \mathbb{S}^{1}\right) =\min \mathcal{N}_{m}\left( \mathbb{S}%
^{1}\right) $. Unlike the case on $\mathbb{S}^{2}$, it is known that%
\begin{equation}
N_{m}\left( \mathbb{S}^{1}\right) =m+1.  \label{eq6.10}
\end{equation}%
Indeed if $N\in \mathcal{N}_{m}\left( \mathbb{S}^{1}\right) $, we must have $%
N\geq m+1$. Otherwise, for the $z_{1},\cdots ,z_{N}\in \mathbb{S}^{1}$ in (%
\ref{eq6.9}), we let $f\left( z\right) =\left( z-z_{1}\right) \cdots \left(
z-z_{N}\right) $, then $\func{Re}f,\func{Im}f\in \mathcal{P}_{m}$. It
follows that%
\begin{equation*}
\frac{1}{2\pi }\int_{\mathbb{S}^{1}}fd\theta =\nu _{1}f\left( z_{1}\right)
+\cdots +\nu _{N}f\left( z_{N}\right) =0.
\end{equation*}%
On the other hand, we clearly have%
\begin{equation*}
\frac{1}{2\pi }\int_{\mathbb{S}^{1}}fd\theta =\left( -1\right)
^{N}z_{1}\cdots z_{N}\neq 0.
\end{equation*}%
This gives us a contradiction. Hence $N_{m}\left( \mathbb{S}^{1}\right) \geq
m+1$. On the other hand, for $1\leq k\leq m+1$, we let $\nu _{k}=\frac{1}{m+1%
}$ and $z_{k}=e^{\frac{2k\pi }{m+1}i}$. It follows that $m+1\in \mathcal{N}%
_{m}\left( \mathbb{S}^{1}\right) $. Hence $N_{m}\left( \mathbb{S}^{1}\right)
=m+1$.

Now we are ready to state the analogue of Theorem \ref{thm1.1} on $\mathbb{S}%
^{1}$.

\begin{lemma}
\label{lem6.1} Assume $m\in \mathbb{N}$, $u\in H^{1}\left( D\right) $ such
that $\int_{\mathbb{S}^{1}}ud\theta =0$ and $\int_{\mathbb{S}%
^{1}}e^{u}e^{ik\theta }d\theta =0$ for $1\leq k\leq m$, then for any $%
\varepsilon >0$ we have%
\begin{eqnarray}
\log \int_{\mathbb{S}^{1}}e^{u}d\theta &\leq &\left( \frac{1}{4\pi
N_{m}\left( \mathbb{S}^{1}\right) }+\varepsilon \right) \left\Vert \nabla
u\right\Vert _{L^{2}\left( D\right) }^{2}+c_{\varepsilon }  \label{eq6.11} \\
&=&\left( \frac{1}{4\pi \left( m+1\right) }+\varepsilon \right) \left\Vert
\nabla u\right\Vert _{L^{2}\left( D\right) }^{2}+c_{\varepsilon }.  \notag
\end{eqnarray}
\end{lemma}

Note that for $m=1$, Lemma \ref{lem6.1} is treated in \cite[lemma 2.5]{OsPS}%
. We can prove Lemma \ref{lem6.1} by replacing (\ref{eq1.2}) with (\ref%
{eq6.2}) and following the approach in Section \ref{sec2} and Section \ref%
{sec3}. The detail is left to interested readers.

To continue we denote%
\begin{equation}
\mathcal{S}_{m}=\left\{ u\in H^{1}\left( D\right) :\int_{\mathbb{S}%
^{1}}ud\theta =0,\int_{\mathbb{S}^{1}}e^{u}e^{ik\theta }d\theta =0\text{ for 
}k=1,\cdots ,m\right\} .  \label{eq6.12}
\end{equation}%
Let $a\in \left( \frac{1}{4\pi \left( m+1\right) },\frac{1}{4\pi m}\right) $%
, then it follows from Lemma \ref{lem6.1} that%
\begin{equation}
\inf_{u\in \mathcal{S}_{m}}\left[ a\left\Vert \nabla u\right\Vert
_{L^{2}\left( D\right) }^{2}-\log \left( \frac{1}{2\pi }\int_{\mathbb{S}%
^{1}}e^{u}d\theta \right) \right]  \label{eq6.13}
\end{equation}%
is achieved.

Let $u$ be a minimizer for (\ref{eq6.13}), then $u$ is smooth and for some
real numbers $\beta _{k}$ and $\gamma _{k}$,

\begin{eqnarray*}
-\Delta u &=&0\text{ in }D; \\
2a\frac{\partial u}{\partial \nu }-\frac{e^{u}}{\int_{\mathbb{S}%
^{1}}e^{u}d\theta } &=&-\frac{1}{2\pi }+\sum_{k=1}^{m}\left( \beta _{k}\cos
k\theta +\gamma _{k}\sin k\theta \right) e^{u}.
\end{eqnarray*}%
Here $\nu $ is the unit outer normal direction of $\mathbb{S}^{1}$. Let%
\begin{equation}
v=u-\log \left( 2a\int_{\mathbb{S}^{1}}e^{u}d\theta \right) ,  \label{eq6.14}
\end{equation}%
then $v$ is smooth and%
\begin{eqnarray*}
-\Delta v &=&0\text{ in }D; \\
\frac{\partial v}{\partial \nu }+\frac{1}{4\pi a} &=&e^{v}+\sum_{k=1}^{m}%
\left( c_{k}e^{ik\theta }+\overline{c_{k}}e^{-ik\theta }\right) e^{v}; \\
\int_{\mathbb{S}^{1}}e^{v}e^{ik\theta }d\theta &=&0\text{ for }k=1,\cdots ,m.
\end{eqnarray*}%
Here $c_{1},\cdots ,c_{m}$ are complex constants. Next we claim $c_{k}=0$
for all $k$. For the case $m=1$, this is proved in \cite[lemma 2.6]{OsPS}.

\begin{lemma}
\label{lem6.2}Let $m\in \mathbb{N}$, $\alpha >0$, $v\in C^{\infty }\left( 
\overline{D}\right) $ such that $\int_{\mathbb{S}^{1}}e^{v}e^{ik\theta
}d\theta =0$ for $k=1,\cdots ,m$ and%
\begin{eqnarray}
-\Delta v &=&0\text{ in }D;  \label{eq6.15} \\
\frac{\partial v}{\partial \nu }+\alpha &=&e^{v}+\sum_{k=1}^{m}\left(
c_{k}e^{ik\theta }+\overline{c_{k}}e^{-ik\theta }\right) e^{v};
\label{eq6.16}
\end{eqnarray}%
here $\nu $ is the unit outer normal direction of $\mathbb{S}^{1}$ and $%
c_{1},\cdots ,c_{m}$ are complex constants, then $c_{k}=0$ for $1\leq k\leq
m $.
\end{lemma}

\begin{proof}
We write%
\begin{eqnarray*}
\left. v\right\vert _{\mathbb{S}^{1}} &=&\sum_{k=-\infty }^{\infty
}a_{k}e^{ik\theta },\quad a_{k}\in \mathbb{C},\overline{a_{k}}=a_{-k}; \\
\left. e^{v}\right\vert _{\mathbb{S}^{1}} &=&\sum_{k=-\infty }^{\infty
}b_{k}e^{ik\theta },\quad b_{k}\in \mathbb{C},\overline{b_{k}}=b_{-k}.
\end{eqnarray*}%
It follows from the assumption that%
\begin{equation}
b_{k}=0\text{ for }1\leq \left\vert k\right\vert \leq m.  \label{eq6.17}
\end{equation}%
Using (\ref{eq6.15}) and (\ref{eq6.16}) we see%
\begin{equation*}
\sum_{k=-\infty }^{\infty }\left\vert k\right\vert a_{k}e^{ik\theta }+\alpha
=\left( 1+\sum_{j=1}^{m}c_{j}e^{ij\theta }+\sum_{j=1}^{m}\overline{c_{j}}%
e^{-ij\theta }\right) \sum_{k=-\infty }^{\infty }b_{k}e^{ik\theta }.
\end{equation*}%
Compare the constant term on both sides and using (\ref{eq6.17}) we get $%
b_{0}=\alpha $. On the other hand, for $k\neq 0$, we have%
\begin{equation}
\left\vert k\right\vert a_{k}=b_{k}+\sum_{j=1}^{m}c_{j}b_{k-j}+\sum_{j=1}^{m}%
\overline{c_{j}}b_{k+j}.  \label{eq6.18}
\end{equation}%
Next we observe that%
\begin{equation*}
\partial _{\theta }\left( e^{v}\right) =e^{v}\partial _{\theta }v,
\end{equation*}%
hence%
\begin{equation*}
\sum_{k=-\infty }^{\infty }kb_{k}e^{ik\theta }=\left( \sum_{j=-\infty
}^{\infty }ja_{j}e^{ij\theta }\right) \left( \sum_{k=-\infty }^{\infty
}b_{k}e^{ik\theta }\right) .
\end{equation*}%
It follows that%
\begin{equation}
kb_{k}=\sum_{j=-\infty }^{\infty }ja_{j}b_{k-j}.  \label{eq6.19}
\end{equation}%
Plug (\ref{eq6.18}) into (\ref{eq6.19}), we get%
\begin{equation*}
kb_{k}=\sum_{j=-\infty }^{\infty }\limfunc{sgn}\left( j\right) \left[
b_{j}+\sum_{s=1}^{m}c_{s}b_{j-s}+\sum_{s=1}^{m}\overline{c_{s}}b_{j+s}\right]
b_{k-j}.
\end{equation*}%
In particular, for $1\leq k\leq m$, it becomes%
\begin{eqnarray*}
kb_{k}
&=&\sum_{j=1}^{k}b_{j}b_{k-j}+\sum_{s=1}^{m}c_{s}%
\sum_{j=1}^{k+s}b_{j-s}b_{k-j}+\sum_{s=1}^{k}\overline{c_{s}}%
\sum_{j=1}^{k-s}b_{j+s}b_{k-j} \\
&&+\sum_{s=k+1}^{m}\overline{c_{s}}\sum_{j=k-s+1}^{0}b_{j+s}b_{k-j}.
\end{eqnarray*}%
Using (\ref{eq6.17}) we get $\alpha ^{2}c_{k}=0$, hence $c_{k}=0$.
\end{proof}

It follows from Lemma \ref{lem6.2} that the function $v$ defined in (\ref%
{eq6.14}) satisfies%
\begin{eqnarray*}
-\Delta v &=&0\text{ in }D; \\
\frac{\partial v}{\partial \nu }+\frac{1}{4\pi a} &=&e^{v}\text{ on }\mathbb{%
S}^{1}.
\end{eqnarray*}%
Since $\frac{1}{4\pi a}\in \left( m,m+1\right) $, it follows from \cite[%
lemma 2.3]{OsPS} that $v$ is a constant function. Hence any minimizer of (%
\ref{eq6.13}) must be $0$. In another word, for any $u\in \mathcal{S}_{m}$,%
\begin{equation*}
\log \left( \frac{1}{2\pi }\int_{\mathbb{S}^{1}}e^{u}d\theta \right) \leq
a\left\Vert \nabla u\right\Vert _{L^{2}\left( D\right) }^{2}.
\end{equation*}%
Let $a\rightarrow \frac{1}{4\pi \left( m+1\right) }$, we get (\ref{eq6.1}).

If $u\in \mathcal{S}_{m}$ such that%
\begin{equation*}
\log \left( \frac{1}{2\pi }\int_{\mathbb{S}^{1}}e^{u}d\theta \right) =\frac{%
\left\Vert \nabla u\right\Vert _{L^{2}\left( D\right) }^{2}}{4\pi \left(
m+1\right) },
\end{equation*}%
then $u$ is smooth and for some real numbers $\beta _{k}$ and $\gamma _{k}$,

\begin{eqnarray*}
-\Delta u &=&0\text{ in }D; \\
\frac{1}{2\pi \left( m+1\right) }\frac{\partial u}{\partial \nu }-\frac{e^{u}%
}{\int_{\mathbb{S}^{1}}e^{u}d\theta } &=&-\frac{1}{2\pi }+\sum_{k=1}^{m}%
\left( \beta _{k}\cos k\theta +\gamma _{k}\sin k\theta \right) e^{u}.
\end{eqnarray*}%
Let%
\begin{equation*}
v=u-\log \frac{\int_{\mathbb{S}^{1}}e^{u}d\theta }{2\pi \left( m+1\right) },
\end{equation*}%
it follows from Lemma \ref{lem6.2} that%
\begin{eqnarray*}
-\Delta v &=&0\text{ in }D; \\
\frac{\partial v}{\partial \nu }+m+1 &=&e^{v}\text{ on }\mathbb{S}^{1}.
\end{eqnarray*}%
By \cite[theorem 7]{Wa}, we can find $\xi \in \mathbb{C}$ with $\left\vert
\xi \right\vert <1$ such that%
\begin{equation*}
v\left( z\right) =\log \frac{\left( m+1\right) \left( 1-\left\vert \xi
\right\vert ^{2}\right) }{\left\vert 1-\xi z^{m+1}\right\vert ^{2}}.
\end{equation*}%
Using the fact $\int_{\mathbb{S}^{1}}ud\theta =0$, we see $u\left( z\right)
=\log \frac{1}{\left\vert 1-\xi z^{m+1}\right\vert ^{2}}$.

At last calculation shows for any $\xi \in \mathbb{C}$ with $\left\vert \xi
\right\vert <1$, if we write $u_{\xi }\left( z\right) =\log \frac{1}{%
\left\vert 1-\xi z^{m+1}\right\vert ^{2}}$, then $u_{\xi }\in \mathcal{S}%
_{m} $ and%
\begin{equation*}
\log \left( \frac{1}{2\pi }\int_{\mathbb{S}^{1}}e^{u_{\xi }}d\theta \right)
=\log \frac{1}{1-\left\vert \xi \right\vert ^{2}}=\frac{1}{4\pi \left(
m+1\right) }\left\Vert \nabla u_{\xi }\right\Vert _{L^{2}\left( D\right)
}^{2}.
\end{equation*}%
Theorem\ \ref{thm6.1} follows.

\end{document}